\documentclass[ 11pt,a4paper, twoside]{amsart}
\usepackage{hyperref}
\usepackage{amsmath,amssymb,amsthm}
\usepackage{algorithm,algorithmicx,algpseudocode}
\usepackage{color}
\usepackage{graphicx}
\newtheorem {proposition}{Proposition}

\newtheorem {remark}{Remark}

\def \N {{\mathbb{N}}}
\def \R {{\mathbb {R}}}

\newcommand{\Ffun}{\mathcal{F}}
\newcommand{\Kfun}{\mathcal{K}}
\newcommand{\Gfun}{\mathcal{G}}

\newcommand{\MAP}{{\rm MAP}}

 \newcommand{\bphi}{{\boldsymbol \varphi}}
 \newcommand{\bA}{{\bf A}}
 \newcommand{\Ao}[1]{{\bf A}_{\al_{#1}}}

\newcommand{\covl}{c_{\Phi}}
\newcommand{\Covl}{C_{\Phi}}
\newcommand{\bCovl}{{\bf C}_{\Psi}}
\newcommand{\bWem}{{\bf W}}
\newcommand{\bW}[1]{{\bf W}(#1)}
\newcommand{\bWeps}[1]{{\bf W}_{\epsilon}(#1)}
\newcommand{\bx}{{x}}

\newcommand{\bP}{{\bf P}}

\newcommand{\al}{ \alpha}
\def\ag{\al_g}

\def\norm#1{\hspace{0.2ex} \|#1\| \hspace{0.2ex}}
\def\oh{\frac{1}{2}}
\def\eps{\varepsilon}

\newcommand{\bd}{{\bf d}}

\newcommand{\br}{{r}}
\newcommand{\Om}{\Omega}
\newcommand{\bPsi}{\boldsymbol{\Psi}}

\def\cetap{C_\eta}
\def\data#1{\varphi_{\al_{#1}}}

\def \bPhi{{\boldsymbol \Phi}}

\def\vcn{{\boldsymbol \rho}}

\newcommand{\myPsi}{\Theta}
\newcommand{\bmyPsi}{\boldsymbol{\Theta}}

\newcommand{\ind}{{m}}

\def\lip{\lambda}
\def \dis{\displaystyle}
\def \diag{\mbox{diag}}

\DeclareMathOperator*{\argmin}{arg\,min}

\renewcommand{\bigotimes}{\mathop{\raisebox{-.5ex}{\hbox{\huge{$\times$}}}}}

\title{Towards analytical model optimization in atmospheric tomography}

\author{Tapio Helin, Stefan Kindermann, Daniela Saxenhuber}

\email{Tapio.Helin@helsinki.fi}
\email{stefan.kindermann@indmath.uni-linz.ac.at}
\email{saxenhuber@indmath.uni-linz.ac.at}

\begin{document}

\maketitle

%\tableofcontents

\begin{abstract}
Modern ground-based telescopes rely on a technology called adaptive optics (AO) in order to compensate for the loss of
image quality caused by atmospheric turbulence.
Next-generation AO systems designed for a wide field of view 
require a stable and high-resolution reconstruction of the refractive index fluctuations in the atmosphere. 
By introducing a novel Bayesian method,  we address the problem of estimating an  atmospheric turbulence
strength profile and reconstructing 
the refractive index fluctuations  simultaneously, where we only use 
wavefront measurements of incoming light from guide stars.
Most importantly, we demonstrate how this method can be used for model optimization as well.
We propose two different algorithms for solving the maximum a posteriori estimate:
the first approach is based on alternating minimization and has the advantage of integrability 
into existing atmospheric tomography methods. 
In the second approach, we formulate a convex non-differentiable optimization problem, 
which is solved by an iterative thresholding method. 
This approach clearly illustrates the underlying sparsity-enforcing mechanism for the strength profile. 
By introducing a tuning/regularization parameter, 
an automated model reduction of the layer structure of the atmosphere is achieved. 
Using numerical simulations, we demonstrate the performance of our 
method in practice.
% 
% %This ill-posed problem is called atmospheric tomography.
% 
% The real-time nature of atmospheric tomography dictates that applicable computational 
% resources are very limited. Added to the notably small tomographic angle, the crux of atmospheric 
% tomography is to design a robust statistical model and efficient discretization of the unknown. 
% 
% Current reconstruction algorithms are based on the assumption that a crucial component of the
% statistical model, so-called turbulence strength profile, is given. 
% 
% This quantity describes the 
% distribution of the (turbulence) energy in the vertical coordinate. 
% 
% The strength profile fluctuates 
% in time and needs to be measured frequently by external instruments. 
% 
% For the future generation telescopes,
% the astronomical community has begun developing algorithms for estimating the strength profile within the AO system. 
% 
% However, with a change in strength profile, the discretization of the atmosphere in the tomographic reconstructor 
% may become inefficient. To the knowledge of the authors, no analytical method exists to optimize the discretization
% (so-called layered model) with respect to the strength profile.
% 
\end{abstract}

\section{Introduction}
In the next generation of telescopes, called the extremely large telescopes (ELT), 
atmospheric turbulence is the major limiting factor for the angular resolution. 
Adaptive optics (AO) systems are designed to improve the imaging quality by 
providing real-time correction for the unwanted distortions generated by the 
atmosphere. Next generation AO systems are required to 
produce a good correction in a large field of view. To achieve this, they use
the measurements of incoming wavefronts from reference light sources 
(guide stars) for the reconstruction of the turbulence (refractive index 
fluctuations) above the telescope. At the core of this challenge is a severely 
ill-posed mathematical problem called \emph{atmospheric tomography}. Based on 
the turbulence profile, the shape of the deformable mirrors (DM) has to be determined 
such that the image of the scientific objects is corrected after reflection on 
the deformable mirrors. 
%
%Modern ground based telescopes rely on a technology called adaptive optics (AO) in order to compensate for the loss of
%image quality caused by atmospheric turbulence. The next generation of telescopes, 
%called the extremely large telescopes (ELT),  require a stable and 
%high-resolution reconstruction of the refractive index fluctuations in the atmosphere. There, 
%atmospheric turbulence is the major limiting factor for the angular resolution. 
%
%Adaptive optics systems are designed to improve the imaging quality by 
%providing real-time correction for the unwanted distortions generated by the 
%atmosphere. Next generation AO systems are required to 
%produce a good correction in a large field of view.
%To achieve this, they use
%the measurements of incoming wavefronts from reference light sources 
%(guide stars) for the reconstruction of the turbulence (refractive index 
%fluctuations) above the telescope.
%
%At the core of this challenge is a severely 
%ill-posed mathematical problem called \emph{atmospheric tomography}. Based on 
%the turbulence profile, the shape of the deformable mirrors (DM) has to be determined 
%such that the image of the scientific objects is corrected after reflection on 
%the deformable mirrors. 

%{\color{red}
The ill-posedness of atmospheric tomography is clear, when considering some problem parameters in 
the European Extremely Large Telescope (E-ELT). In the E-ELT, 9 guide stars and associated wavefront sensors (WFSs) 
sample a field of view of at most 10 arcmin in diameter (e.g., Multi-Object 
Adaptive optics \cite{Rousset10}), while the relevant turbulent atmosphere at 
the site reaches 
the altitude of roughly 20 km \cite{Vernin11}. It goes without saying that the 
achievable vertical resolution is low and the  model assumptions  play a 
crucial role in the regularization of the tomography problem.

Current reconstruction algorithms are based on 
the so-called layered 
model \cite{RoWe96} where the atmosphere is divided into a finite number of 
slabs or layers  (see Fig.~\ref{fig:atmtomo}). 
The key 
assumptions are that the turbulence subregions have approximately homogeneous 
statistics and that each layer is assumed to be statistically independent. 
A crucial component of all layered models is the 
\emph{refractive index structure parameter} (also known  as  $C_n^2$-profile or strength profile
in  telescope imaging literature). % is given. 
The strength profile describes the distribution of the (turbulence) energy 
in the vertical coordinate.  It  fluctuates in time and needs to be measured frequently by external 
instruments. 
%For the future generation telescopes, the astronomical community has begun developing algorithms 
%for estimating the strength profile within the AO system. 

However, with a change in the strength profile,
the discretization (e.g., the number and position of the layers) of the atmosphere in the tomographic reconstructor may 
become inefficient. As the 
real-time nature of atmospheric tomography dictates that applicable computational resources are very limited, 
an optimized model is vital  in achieving good imaging quality.
This remains an issue  as well in the next-generation implementations due to 
the rapid increase of system sizes.
To the knowledge of the authors, no analytical method exists to optimize the discretization 
with respect to the strength profile.

While the literature on reconstruction methods is well-developed, analytical 
model optimization has gained less attention: how sensitive are the algorithms 
to fluctuations in the strength profile? 
 How to choose an optimal model for the atmosphere?
In the current generation of telescopes, the turbulence strength profile is 
measured using independent instruments, e.g., MASS and 
SCIDAR~\cite{Masciadri14}. The impact of the 
turbulence strength on the tomographic reconstruction has been studied, e.g., in 
\cite{FuCo10,CoFu12,Gendron14,Costille12}. Novel numerical methods have been 
recently proposed for estimating the profile based on the wavefront sensor 
measurements \cite{gilles2010real, cortes2012}. Such a method would carry the 
advantage of integrating the strength profile modelling more closely to the 
reconstruction algorithm. Nonetheless, the question of choosing an optimal 
computational model seems open.

In this work we focus on  model optimization in atmospheric tomography. We 
introduce a novel reconstruction method that simultaneously produces an estimate 
of the turbulent layers as well as the turbulence profile. We show that while 
such an algorithm requires a comparably large computational effort, it can be qualitatively 
superior in situations where there is uncertainty in the strength profile. What 
is more, we present a modification of this method which enforces sparsity on 
the profile-solution. In other words, the algorithm produces an estimate of the 
atmosphere which is optimized in both the data fidelity as well as the number 
and the altitudes of the layers. The ability to do so is the crux of model 
optimization. Given the computational resources (roughly how many layers can be 
modeled), our method is able to optimize the model (altitudes of the layers) 
based on the data-stream.

For the optimization, we propose two alternative procedures, an alternating minimization-type algorithm, 
which switches between a layer-reconstruction step and a layer-strength identification step 
and an iterative shrinkage-type method, which promotes a sparse atmosphere model and 
hence is tailored for  model reduction. Although both methods approximate the solution to the 
same optimization problem, they have different scopes, as we will explain below. 

A second major contribution in this paper is to include an additional clustering structure 
of the layers by grouping them into independent clusters.  Such a step is important in 
view of an intended  model reduction because otherwise the predominant ground layer 
overshadows all other layers.

The results in this paper are demonstrated and confirmed by simulations in an adaptive 
optics simulation environment called the MOST \cite{Au15}. Let us mention that for 
the simulations, we have chosen rather idealistic imaging conditions. 
%We assume that the WFS measurements contain little noise, i.e., we have so-called high-flux imaging setting. Moreover, we neglect two phenomena that are related to modelling a laser guide star: 
Including more practically relevant  effects such as  spot elongation and  tip-tilt 
indetermination \cite{RoWe96} is important if one aims to measure the real-life performance of a 
reconstruction algorithm. In this work we instead concentrate on the effects taking place in the 
reconstructed strength profile and how energy is divided to different altitudes. Our objective is to 
build the  basis for such an optimization algorithm.

%In general, read this carefully \cite{guesalaga2015lessons}

This paper is organized as follows. In Section \ref{sec:problem_setting}, we describe the problem 
setting and review a model for atmospheric turbulence. Section \ref{sec:bayes} is concerned with a 
Bayesian approach to the problem of atmospheric tomography. In Section \ref{sec:algo}, we state several equivalent 
optimization problems
and propose two different algorithms, an alternating minimization procedure and an iterative shrinkage method. 
Numerical results are presented in Section \ref{sec:numerics}.

\section{Problem setting} \label{sec:problem_setting}

%---------------------------------
\subsection{Light propagation in the atmosphere}
%---------------------------------

The wind in the atmosphere causes an irregular mixing of warm and cold air. This effect is
called the atmospheric turbulence. The fluctuations of the temperature are essentially
proportional to the refractive index fluctuations \cite{Ro99}, and hence,  turbulence affects the propagation of light. 

In atmospheric tomography one aims to reconstruct the turbulence profile given wavefront sensor (WFS) measurements. Further, 
in an AO system, these data are used to adapt the shape of one or more deformable mirrors (DM).
Below we assume that the tomography problem (WFS data to atmosphere reconstruction) can be divided into two
reliable steps where the incoming wavefronts are resolved in  an intermediate stage. Together with the DM shape optimization step,
such an approach is sometimes referred to as the 3-step approach \cite{RaRo12}. 
The 3-step approach typically requires good imaging conditions in order to be successful.
 
\begin{figure}[!ht]
\centering
\includegraphics[scale=0.25]{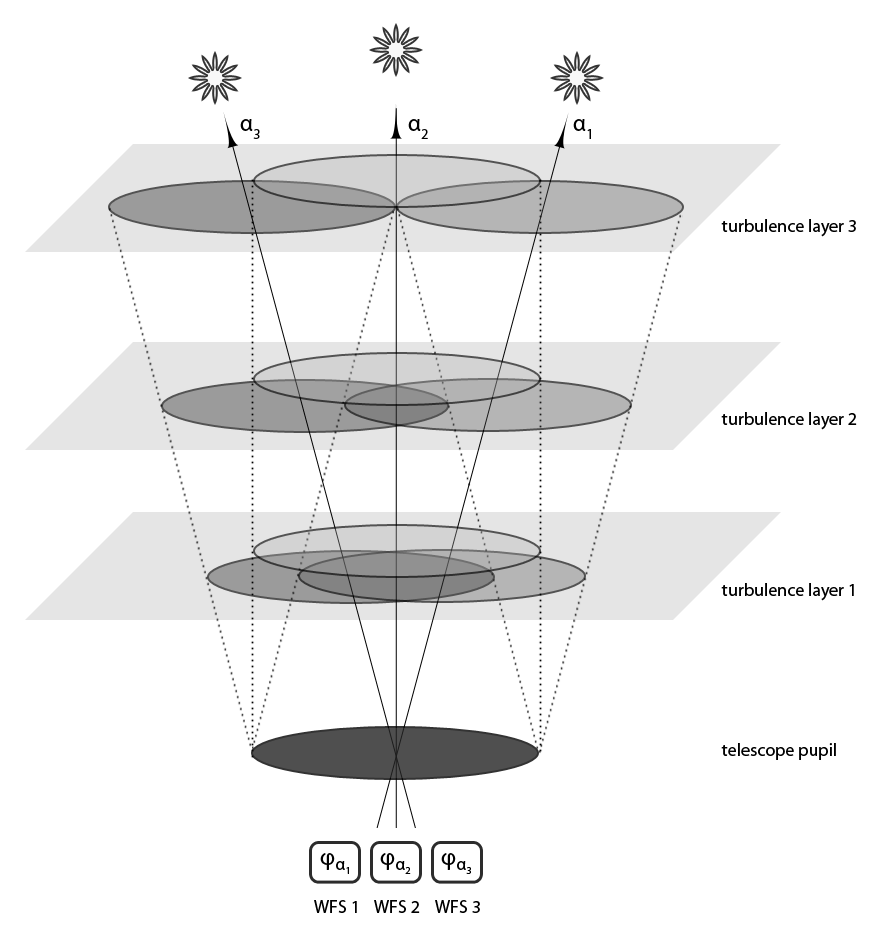}
\caption{Illustration of the atmospheric tomography problem with three guide 
star directions, three WFS, and three atmospheric layers.}
\label{fig:atmtomo}
\end{figure}

In the following, we are mainly interested in the second step, i.e., estimating the atmosphere from the incoming wavefronts (see Fig,~\ref{fig:atmtomo}). 
Let us therefore explain next the physical model that connects the
turbulence layers located at  
different  altitudes $h_l$, $l=1,\ldots L$, to the incoming wavefronts $\varphi_{\ag}$, $g=1,...,G$, in each guide star direction.

We define the forward operator $\bA$ as the mapping from the atmosphere
consisting of $L$ layers to the $G$ incoming wavefronts
\begin{align}
\label{eq:forward_operator}
\begin{split}
\bA :\bigotimes_{l=1}^L L^2(\Om_l) \, \to \, \bigotimes_{g=1}^GL^2(\Om_D)\,, \\
\bA =: [\Ao{1},\dots , \Ao{G}]\,,
\end{split}
\end{align}
with the operators in guide star directions $\alpha_g:= (\alpha_g^{(1)}, \alpha_g^{(2)}),$
corresponding  to the 3D directional vector $(\alpha_g^{(1)}, \alpha_g^{(2)},1),$  $g= 1,\ldots G.$
Above, $\Om_D$ and $\Om_l$   stand for the layer domains and the telescope aperture, respectively;
see \eqref{eq:omd}, \eqref{eq:oml}.
A next-generation adaptive optics system typically utilizes a mix of  \textit{natural guide stars (NGS)} and
\textit{laser guide stars (LGS)}. For this work, the crucial difference is in the geometry of light propagation explained below.

Under the geometric optics approximation and appropriate assumptions on the atmosphere, the distortions in the phase of
light are proportional to the integral over the refractive index fluctuations along the path of light.
In the layered model of atmosphere, the precise formula for the forward operator is
\begin{align}
\label{eq:aag}
\begin{split}
\Ao{g}:&\bigotimes_{l=1}^L  L^2(\Om_l)  \to L^2(\Om_D) \,, \\
\Ao{g} \bPhi &:=\sum_{l=1}^L  {\Phi^{(l)}(c_l\br+h_{l}\al_{g})} =\data{g}(\br)~,\,  \quad \br \in \Om_D\, , g = 1,\ldots G,
\end{split}
\end{align}
where $\Phi^{(l)}$ represents the fluctuations of the refractive index (also referred to as ``turbulence'' in the following) at 
layer $l$. 
Here $\ag$ corresponds to the 3D directional vector $(\alpha_g^{(1)}, \alpha_g^{(2)},1)$ and 
$c_l$ is a factor needed in the case of  laser guide stars. 
In fact, the conical propagation of light in LGSs requires a scaling factor 
\begin{align}
c_l:= \begin{cases} 1\,, & \mbox{for NGS}\\ 1- \frac{h_l}{h_{LGS}}, & \mbox{for LGS and }\, 
h_{LGS}. \end{cases}
\end{align}
Above, $h_{LGS}$ denotes the altitude where the laser beam scatters, i.e., the altitude where the artificial light source is observed.
We assume that the telescope aperture is annular with \mbox{$D>d\geq 0$} and
\begin{align}\label{eq:omd}
\Om_D := \{ \br = (x,y) \in \R^2\,:\, d \leq \norm{r} \leq D\}.
\end{align}
In consequence, the visible areas of the atmosphere at layer $l$ can be described as
\begin{align}\label{eq:oml}
\Om_l := \bigcup_{g=1}^G \Om_D(h_l \ag)\,, \quad  \mbox{ with } \, \quad \Om_D(h_l \ag) :=
\left\{ \br \in \R^2 \, : \, \frac{\br  -h_l \ag}{c_l} \in \Om_D\right\}\,.
\end{align}
Note that it is also possible to consider the tomography operator \eqref{eq:aag} acting 
between the spaces $\dis\bigotimes_{l=1}^L  H^1(\Om_l)$ to $H^1(\Om_D)$; see \cite{EsPeRa13}.
Let us explain our notation:  we write  functions representing turbulence (and only those) 
at some layer $l$ with an upper bracketed index, $\Phi^{(l)}.$  
Moreover, to distinguish between objects on layers and the collection of these 
objects into a vector, we write the latter as well as all operators acting on these by bold symbols.  
%Moroever, we write relevant vectors and operators in boldface.   
Thus, the total atmospheric turbulence
as the vector having the turbulence at the different layers as entries  is  expressed as $\bPhi = (\Phi^{(1)},\ldots, \Phi^{(L)})$. 
%Moroever, vectors of 
%layerwise objects  and also operators acting on such vectors are repesented 
%by bold symbols. 

%---------------------------------
\subsection{Atmospheric turbulence}
%---------------------------------

Statistical models for turbulence are frequently utilized in the AO literature when postulating 
the tomography step as a Bayesian inference problem.
The classical works by Kolmogorov \cite{Ko41} suggest that the turbulence statistics 
can be modeled by a homogeneous and isotropic Gaussian random field. The key assumption is 
that the power spectral density satisfies Kolmogorov's power law, i.e.,
it is proportional to $|\kappa|^{-11/3}$ for $\ell_0 \leq \kappa \leq L_0$ where $\kappa$ is the spatial 
frequency of the turbulence field and  the bounds $\ell_0$ and $L_0$ (inner and outer scale, respectively) 
define the so-called inertial range.
Deviations from Kolmogorov turbulence have been debated, but currently Kolmogorov models or 
variants are commonly used in the reconstruction algorithm literature related to atmospheric tomography.

In this work we assume a so-called von Karman statistics \cite{vonKarman} that modifies the Kolmogorov model in order
to avoid the 
singularity at $\kappa = 0$.
Under the von Karman model, the cumulative turbulence integrated over the atmosphere has a covariance function
\begin{equation}
\label{eq:cphi}
\covl(\bx_1,\bx_2) = \covl(\Delta \bx) =   \left(\tfrac{L_0}{r_0}\right)^\frac{5}{3}\tfrac{c}{2}
\left(\tfrac{2\pi\Delta x}{L_0}\right)^\frac{5}{6} K_{5/6} \left(\tfrac{2\pi \Delta \bx}{L_0}\right)\,,
\end{equation}
where $\bx_1,\bx_2\in \R^2,$ $\Delta x = |x_1-x_2|$, the constant 
$c = \frac{2^{1/6}\Gamma(11/6)}{\pi^{8/3}}\left(\frac{24}{5} \Gamma(\frac{6}{5})\right)^{5/6}$, 
$K(.)$ is the modified Bessel function of the second type \cite{AbrSteg}, 
$\Gamma$ is the gamma function, and $r_0$ is the Fried parameter~\cite{BeCeMa08}.
 Notice carefully that since $\Phi$ is a stationary random field, 
 the covariance function depends only on the separation of the points $\bx_1$ and $\bx_2$.

Now let $\Covl$ denote the covariance operator formally induced by $\covl$ in equation \eqref{eq:cphi}. We remark 
that the realizations from the probability distribution of $\Phi$ do not decay at infinity and thus fail to belong to 
typical function spaces, e.g., $L^2(\R^2)$. The definition of $\Covl$ can be made precise as a pseudodifferential
operator of order $-11/3$ in weighted Sobolev spaces or using the formalism of generalized random variables \cite{Gelfand}. 
For a detailed expression of $\Covl$, see \cite{HeYu13}.
In the following we neglect a more detailed analysis and work with the discretized version of $\Covl$.

In the layered model of the atmosphere, the cumulative turbulence  is decomposed into  separate layers by assuming that 
the a priori distribution of $\bPhi =(\Phi^{(1)},\ldots,\Phi^{(L)})$ has zero-mean Gaussian statistics with covariance operator ${\rm diag}(\rho_1 \Covl, ..., \rho_L \Covl)$,
where the relative turbulence strength vector $\vcn:= (\rho_1,\ldots \rho_L )$ is normalized, i.e.,
\begin{equation}\label{normalization}
\sum_{l=1}^L \rho_l = 1.
\end{equation}
Recall that in adaptive optics the cumulative turbulence is typically given from independent measurements and expressed in equation \eqref{eq:cphi}. 
Therefore, the normalization is required for $\vcn$.

\section{The Bayesian approach} \label{sec:bayes}
In atmospheric tomography one considers the problem
\begin{equation}
\label{eq:tomography_problem}
\bA \bPhi = \bphi
\end{equation}
where $\bPhi = (\Phi^{(l)})_{l=1}^L$ and $\bphi = (\varphi_{\ag})_{g=1}^G$ represent the unknown turbulence layers and incoming wavefronts, respectively, and $\bA$ is given by equation \eqref{eq:forward_operator}. 
Let us shortly outline how the problem \eqref{eq:tomography_problem} can be interpreted using Bayesian inference.

To avoid technicalities, we assume that the model \eqref{eq:tomography_problem} is discretized by orthogonal projections
\begin{equation*}
	Q_m : \bigotimes_{l=1}^L L^2(\Om_l) \to X_m \subset \bigotimes_{l=1}^L L^2(\Om_l)
	\quad {\rm and} \quad
	R_n : \bigotimes_{g=1}^GL^2(\Om_D) \to Y_n \subset \bigotimes_{g=1}^GL^2(\Om_D),
\end{equation*}
where the parameters $m,n\in \N$ indicate the dimension of each projection, respectively.
We set $\bphi_m = Q_m\bphi$ and $\bPhi_n = R_n \bPhi$ and obtain a linear system
$ \bA_{mn} \bPhi_n=\bphi_m$ with $\bA_{mn} = Q_n \bA R_n$.
For this presentation, the details of the discretization basis are not relevant  and are omitted.

For the rest of the paper, let us abuse the notation by writing $\bA$, $\bPhi$ and $\bphi$ instead of $\bA_{mn}, \bPhi_n$ and $\bphi_m$ and identify them 
with matrices and vectors   in 
Euclidean spaces equipped with the Euclidean norm. Hence, in the following, the  appearing covariance operators
($\cetap$, $\Covl$ below)  can  be represented 
by non-singular matrices. 

Our starting point for the Bayesian setup is based on the splitting of the covariance for $\bPhi$ into  
a relative turbulence strength part and a standardized covariance as above. 
Hence, we  assume that the prior model for $\bPhi$ can be decomposed as
\begin{equation}
	\label{eq:reparametrization_of_phi}
	\bPhi = \bW{\vcn} \bPsi,
\end{equation}
with two mutually independent random components 
\begin{equation}\label{eq: defWpsi}
\bW{\vcn} = {\rm diag}(\sqrt{\rho_j})_{j=1}^L \in \R^{L\times L} \quad \text{ and } \quad 
\bPsi = (\Psi^{(j)})_{j=1}^L  \, . 
\end{equation}
%{\color{red}Analogous to \eqref{eq:layeredmodel}}, we interpret $\vcn$ as the strength of the turbulence on each layer. 

The vector $\vcn$ is assumed to be distributed according to
\begin{equation}\label{eq:hyper1} \pi_{prior}({\vcn}) \sim 
\delta_{\Gamma}(\vcn),
\end{equation}
where $\delta_{\Gamma}$ denotes the Dirac delta on the surface
$\Gamma = \{\vcn \; | \; \sum_{l=1}^L \rho_l = 1, \rho_l\geq 0\} \subset \R^L$. In other words, $\vcn$ is 
assumed to be equally distributed within an $L$-dimensional simplex, i.e., we only include the information 
that it is nonnegative and normalized by \eqref{normalization}.
Furthermore, we assume $\Psi \sim {\mathcal N}(0, \frac 1 \alpha \bCovl)$, where 
$\bCovl = {\rm diag}(\Covl, ..., \Covl)$ and $\alpha>0$ is the model tuning parameter. We discuss the role of $\alpha$ in more detail below.

In the following, we choose to make statistical inference on the pair $(\vcn, \bPsi)$ instead of $\bPhi$. Since the two were assumed 
to be independent, we have
\begin{equation*}
	\pi_{\rm prior}(\vcn, \bPsi) = \pi_{\rm prior}(\vcn) \pi_{\rm prior}(\bPsi).
\end{equation*}
Consequently, the posterior takes the following form
\begin{equation*}
	\pi_{\rm post}(\vcn, \bPsi \; |\; \bphi) =
	c \pi_{\rm prior}(\vcn) \pi_{\rm prior}(\bPsi) \pi_{\rm noise}(\bphi-\bA \bW{\vcn} \bPsi),
\end{equation*}
where the constant $c$ depends on the measurement $\bphi$.
%Let us then set some notation with
%\begin{equation}\label{eq:defL} 
%{\bf L}( \vcn):=
% \left(\bW{\vcn}^2 \bCovl\right)^{-\frac{1}{2}}.
%\end{equation}
Assuming that the measurement noise is Gaussian to a good approximation with  zero mean and covariance matrix $\cetap,$ 
we conclude that the MAP estimate is obtained from a constrained optimization problem 
\begin{equation}
	\label{eq:global_problem}
	\begin{split}
(\bPsi_{\MAP},\vcn_{\MAP}) & \in   \argmin_{\substack{\bPsi \in Y_n \\ \vcn\in \Gamma}} \Kfun(\bPsi,\vcn) \\
\Kfun(\bPsi,\vcn)&:= \norm{\cetap^{-\oh}({\bf A} \bW{\vcn} \bPsi - \bphi)}^2_2 + 
 \alpha \norm{\bCovl^{-1/2} \bPsi}^2_2\, .
 \end{split}
\end{equation}

\begin{remark}
We point out that, in general, a MAP estimate is not invariant under  reparametrizations (see e.g. \cite{druilhet07}). 
More precisely, consider the mapping $f : (\bPhi,\vcn) \mapsto (\bPsi,\vcn)$. In our case,
the image $f(\bPhi_{MAP},\vcn_{MAP})$ does not coincide with $(\bPsi_{MAP},\vcn_{MAP})$ since the dependency of $\bPhi$ and $\vcn$ will introduce additional non-quadratic terms in the minimization problem \eqref{eq:global_problem}. Our numerical simulations suggest that the MAP estimator in \eqref{eq:global_problem} represents the posterior well and has good approximation properties in practical settings. In addition, the MAP estimate becomes discretization invariant with respect to $\bPsi$ \cite{HB15} and connects directly to an interesting class of optimization problems related to sparse structures. It remains (an important) part of future work to study the posterior from perspective of conditional mean estimates and confidence intervals. These practically motivated objectives are out of the scope of this treatise.
\end{remark}

Although not obvious at first sight, the estimator in \eqref{eq:global_problem} has 
a built-in sparsity-enforcing mechanism. In practice, when the dimension of the data is radically 
lower than the unknown, such a method yields an estimate $\vcn_{\MAP}$ which has zero turbulence
strength at several altitudes. This contradicts with reality in the sense that the turbulence profile is 
considered smooth with respect to the altitude. However, we will demonstrate below that decreasing the dimension
(fewer layers are used in the model) reduces or even removes the sparsity effect. 

For a higher number of layers,
our simulations suggest that the solution is optimal under a constraint that a fixed number of turbulence
strength components $\rho_l$ are non-zero. Let us make this claim more precise: suppose we model $L$ layers. 
Then the solution $(\Psi_{MAP},\vcn_{\MAP})$ in \eqref{eq:global_problem} has $L_0$ 
non-zero components $\rho_l$ for some $L_0\leq L$. Moreover, $(\Psi_{MAP},\vcn_{\MAP})$ is optimal (in the sense of minimizing the  functional $\Kfun$) among all reconstructions that have
only $L_0$ non-zero layers (out of the chosen set of $L$ layers). What is more, we can control 
the number $L_0$ by adjusting the regularization term in \eqref{eq:global_problem} by tuning~$\alpha$. %We emphasize that here we support this claim by numerical evidence.

Below we find that in such an approach the layers at lower altitudes become heavily preferred. This is due to both geometrical factors as well as the fact that most of the turbulence is located close to the ground. For a more detailed analysis, see Section \ref{sec:numerics}. %For systems with multiple deformable mirrors the fitting step requires additional consideration.  (EXPLAIN THE METHOD WHERE LAYERS ARE PROJECTED DOWNWARDS)
In order to include layers at higher altitude to the solution we introduce additional information. We assume that we know the 
cumulative turbulence strength on some altitude intervals in the atmosphere. In a nutshell, the set of layers is divided into clusters of layers and 
we know the cumulative strength over 
each cluster.

Let us discuss the clustering in more detail. We assume that  $N< L$ clusters, ${\mathcal B}_1, \dots ,{\mathcal B}_N$, are given which constitute a partition 
of the layers, i.e.,
\begin{equation}
{\mathcal B}_i \subset \{1,\ldots, L\}, \quad 
\bigcup_{i=1}^N \mathcal B_i = \{1,\ldots, L\}, \quad {\mathcal B}_i \cap \mathcal B_j = \emptyset  \text{ for  } i\neq j,
\end{equation}
and on each cluster we fix a total turbulent energy $d_i$ as 
\begin{equation} 
\label{eq:d_i}
d_i:= \sum_{l \in {\mathcal B}_i} \rho_l, \qquad \bd := (d_1,\ldots,d_N),  \qquad \sum_{i=1}^N d_i = 1\,, 
\end{equation}
where $\bd$ is given. 
Further, we assume that the prior is uniformly distributed on the simpleces formed by the 
clusters ${\mathcal B}_j$ in the following way
\begin{align}\label{eq:hyper2} \pi_{\rm prior}({\vcn}) \sim 
\delta_{\Gamma_N}, 
\end{align}
where 
$$\Gamma_N(\bd) = \left\{\vcn \in \R^L \Big| \; \sum_{l \in {\mathcal B}_i} \rho_l = d_i \textrm{ for all } i=1,...,N,    
\rho_l \geq 0, l = 1,\ldots L \right\}. $$
Including this additional information we propose now a 
modified method with an a priori fixed vector $\bd$ such that  the MAP estimate is given 
analogous to \eqref{eq:global_problem}
by 
\begin{equation}
	\label{eq:global_problem2}
(\bPsi_{\MAP},\vcn_{\MAP}) \in  
\argmin_{\substack{\bPsi \in Y_n \\ \vcn\in \Gamma_N(\bd)}} \Kfun(\bPsi,\vcn), 
%
%\left(\norm{\cetap^{-\oh}({\bf A} \bW{\vcn} \bPsi - \bphi)}^2_2 + 
% \alpha \norm{\Covl^{-1/2} \bPsi}^2_2\right)
\end{equation}
with $\Kfun(\bPsi,\vcn)$ as in \eqref{eq:global_problem}.
Obviously, the modified method reduces to the original 
by setting $N = 1$.  

Let us now return to the role of parameter the $\alpha$. 
As mentioned above, the constraint $\vcn\in \Gamma_N$ enforces some sparsity on the solution. 
This phenomenon is illustrated in Proposition \ref{propmin} in more clarity. Moreover, 
we notice that the parameter $\alpha$ controls how sparse the obtained  solutions are. 
From a Bayesian point of view, $\alpha$ defines the subjective demand for model reduction. The parameter $\alpha$ can also be interpreted otherwise: 
consider two problems with parameters $\alpha$ and $\alpha'$.
It is straightforward to see that the problem \eqref{eq:global_problem2} with $\alpha$ is equivalent to a 
situation where $\alpha'=1$, but where $\Gamma_N'$ is normalized to $1/\alpha$, i.e., $\sum_{i=1}^N d_i' = 1/\alpha$.
The intuition is that a large $\alpha$  artificially enforces
a lower  cumulative strength  in order to achieve economically a reasonable model fit. 
In Section~\ref{sec:numerics} we explore the proposed method with various choices of $\alpha$.

%=======================================================
\section{Optimization problems and algorithms}\label{sec:algo}
%=======================================================

\subsection{The maximum a posteriori estimate revisited}

In this section we study problem \eqref{eq:global_problem2}, rewrite it into  alternative forms, and 
discuss algorithms to solve it. We propose two different algorithms, namely an  
alternating minimization procedure and an iterative shrinkage method based on eliminating 
the strength vector $\vcn$. 

\begin{proposition}\label{propmin3}
The optimization problem  \eqref{eq:global_problem2} is 
equivalent to the problem 
\begin{equation}
	\label{eq:global_problem3}
	\begin{split} 
(\bPhi_{\MAP},\vcn_{\MAP}) &\in  
\argmin_{\substack{\bPhi \in Y_n \\ \vcn\in \Gamma_N(\bd)\\ \bW{\vcn}^\dagger \bW{\vcn} \bPhi = \bPhi}} \Ffun(\bPhi,\vcn)  \\
 \Ffun(\bPhi,\vcn)& := \norm{\cetap^{-\oh}({\bf A}  \bPhi - \bphi)}^2_2 + 
 \alpha \norm{\bCovl^{-1/2}  \bW{\vcn}^\dagger \bPhi}^2_2,
 \end{split} 
\end{equation}
where $\bW{\vcn}^\dagger$ denotes the Moore-Penrose pseudoinverse, 
\[ \bW{\vcn}^\dagger =  {\rm diag}\left( \begin{cases} 
\frac{1}{\sqrt{\rho_j}} & \text{ if } \rho_j >0, \\
 0 & \text{ else } \end{cases} \right)_{j=1}^L   \in \R^{L\times L}.  \]
More precisely,  equivalence holds in the following sense: if $(\vcn_{\MAP},\bPsi_{\MAP})$ solves   \eqref{eq:global_problem2}, then 
 $(\vcn'_{\MAP},\bPhi_{\MAP}) :=(\vcn_{\MAP}, \bW{\vcn} \bPsi_{\MAP})$ 
 solves  \eqref{eq:global_problem3}, and if $(\vcn'_{\MAP},\bPhi_{\MAP}) $ solves  \eqref{eq:global_problem3},
 then $(\vcn_{\MAP}, \bPsi_{\MAP}):= (\vcn'_{\MAP},\bW{\vcn'_{\MAP}}^\dagger \bPhi_{\MAP})$ solves  \eqref{eq:global_problem2}.
\end{proposition}
\begin{proof}
Let 
$(\vcn_{\MAP},\bPsi_{\MAP})$ solve   \eqref{eq:global_problem2}. Let ${\mathcal I} = \{i\,|\, (\vcn_{\MAP})_i = 0\}$ be the zero indices of 
$\vcn_{\MAP}$. Since $(\Psi_{\MAP}^{(i)})_{i\in {\mathcal I}}$ do not contribute to the first term in  \eqref{eq:global_problem2} and 
the second term is a sum of squares over the layers, it follows by the minimization property that the corresponding entries in 
$\bPsi$ must vanish: $(\Psi_{\MAP}^{(i)})_{i\in {\mathcal I}} = 0.$ Therefore, problem \eqref{eq:global_problem2} is equivalent to the same 
optimization problem with the additional constraint $\bW{\vcn}^\dagger \bW{\vcn} \bPsi = \bPsi$.  
Hence, by setting 
$\bPhi_{\MAP}:= \bW{\vcn} \bPsi_{\MAP},$ we arrive at  \eqref{eq:global_problem3} with the additional constraint
that there exists a $\bPsi$ with $\bPhi = \bW{\vcn} \bPsi$ and $\bW{\vcn}^\dagger \bPhi = \bPsi$. It is a simple calculation that this 
constraint is satisfied if and only if $\bW{\vcn}^\dagger \bW{\vcn} \bPhi = \bPhi.$ The same argument reversed implies that 
solutions to \eqref{eq:global_problem3}  yield corresponding  solutions to \eqref{eq:global_problem2}. 
\end{proof}

Next, we eliminate the strength vector $\vcn$ from the optimization problem by minimizing over it. 
\begin{proposition}\label{propmin}
Let $(\bPhi_{\MAP},\vcn_{\MAP}) \in Y_n \times \Gamma_N(\bd)$ be minimizers of \eqref{eq:global_problem3}
for some $\bd>0$. Then $\bPhi_{\MAP}$ is also a solution of the 
optimization problem 
\begin{equation}\label{eq:reduced}
\begin{split} 
\bPhi_{\MAP} &= \argmin_{\bPhi \in Y_n} \Gfun(\bPhi), \\
\Gfun(\bPhi)&:= \norm{\cetap^{-\oh}({\bf A} \bPhi - \bphi)}^2_2 + \alpha 
\sum_{i=1}^N \frac{1}{d_i} \left(\sum_{l\in {\mathcal B}_i} \norm{C_{\Phi}^{-1/2}\Phi^{(l)}}\right)^2, 
\end{split}
\end{equation}
where $d_i$ are as in equation \eqref{eq:d_i}. If for some cluster ${\mathcal B}_i,$ the vector $(\bPhi_{\MAP}^{(l)})_{l\in {\mathcal B}_i}$ is 
non-zero, the $l$-th component of 
the strength vectors at this cluster,  $(\vcn_{\MAP})_l,$ is uniquely defined by 
\begin{equation}\label{eq:vcnsol}  \qquad (\vcn_{\MAP})_l = d_i
\frac{\norm{C_{\Phi}^{-1/2}\Phi_{\MAP}^{(l)}}}{\sum_{l \in {\mathcal B}_i} \norm{C_{\Phi}^{-1/2}\Phi_{\MAP}^{(l)}} }, 
\quad \forall l \in {\mathcal B}_i. \end{equation}
\end{proposition}
\begin{proof}
With $\bPhi_{\MAP}$ a solution to \eqref{eq:global_problem3}, we can compare the function values of 
$(\bPhi_{\MAP},\vcn_{\MAP})$ with all admissible values for $(\bPhi_{\MAP},\vcn),$
i.e., where $\vcn$ satisfies the constraints $\vcn \in  \Gamma_N(\bd)$ and $\rho_l >0$ whenever 
$\Phi_{\MAP}^{(l)} \not = 0.$ 
%
%
%those of $(\bPhi_{\MAP},\vcn),$  where $\vcn \in  \Gamma_N(\bd)$ and 
%$\rho_i >0$ whenever $\Phi_{\MAP}^{(i)} \not = 0$.
Since $\vcn_{\MAP}$ must be minimal, it follows that 
\[ \vcn_{\MAP} \in  \argmin_{\vcn \in \Gamma_n(\bd)}  
\sum_{i=1}^N \sum_{l\in {\mathcal B}_i, \Phi^{(l)} \not = 0} \frac{1}{\rho_l}\norm{C_{\Phi}^{-1/2}\Phi_{\MAP}^{(l)}}^2
 \]
since the first term in \eqref{eq:global_problem3} does not depend on $\vcn$. 
This optimization problem can  be solved yielding \eqref{eq:vcnsol} 
for those clusters ${\mathcal B}_i$  where $\Phi_{\MAP}^{(l)}$ is not all zero for $l \in {\mathcal B}_i$. 
For the other clusters the functional value and the constraints  are independent of the strength vector on these
clusters, hence $(\vcn_{\MAP})_l$ can be arbitrary with the only restriction that 
$ \sum_{l \in {\mathcal B}_k} \rho_l= d_k$ on these clusters.  By plugging in the formula for $\vcn_{\MAP}$ 
in  \eqref{eq:vcnsol}  into 
the functional in \eqref{eq:global_problem3}, we find that its optimal value is that of 
$\Gfun(\bPhi_{\MAP})$ in  \eqref{eq:reduced}. Conversely, for any $\bPhi$, we can construct a $\vcn$ 
as given in  \eqref{eq:vcnsol} such that $\Gfun(\bPhi)$ equals $\Ffun(\bPhi,\vcn)$ in  \eqref{eq:global_problem3}. Thus, it follows that 
$\bPhi_{\MAP}$ must be a minimizer of $\Gfun$ as well and this proves the proposition.  
\end{proof}

From these results we can study existence and uniqueness of minimizers of the  functionals above.  
\begin{proposition}
Let $\alpha, \bd>0$, matrices $\bCovl$ and $\cetap$ be positive definite and $(\bPhi,\vcn) \in Y_n \times \Gamma_N(\bd).$ Then problems 
\eqref{eq:global_problem2}, \eqref{eq:global_problem3} and \eqref{eq:reduced}  have 
minimizers. The minimizer $\bPhi_{\MAP}$ in  \eqref{eq:global_problem3} and \eqref{eq:reduced}  is unique 
and  $\vcn_{\MAP}$ is unique at a cluster ${\mathcal B}_i$  if  $(\bPhi_{\MAP}^{(l)})_{l\in {\mathcal B}_i} \not = 0.$ 
\end{proposition}
\begin{proof}
Since the problem is finite dimensional, it follows by 
straightforward argumentation that \eqref{eq:global_problem2} has a solution. Note that 
the functional $\Kfun$ is continuous in both arguments. By Propositions~\ref{propmin3}, \ref{propmin} 
the same holds for \eqref{eq:global_problem3} and  \eqref{eq:vcnsol}. 
The functional $\Gfun$ in \eqref{eq:reduced} is strictly convex in $\bPhi$, thus the minimizer  $\bPhi_{\MAP}$ is unique, 
by  formula \eqref{eq:vcnsol}, $\vcn_{\MAP}$ is unique  if $\bPhi_{\MAP}$ is not all zero at a cluster. 
\end{proof}

The reduced problem \eqref{eq:reduced} can be interpreted as a regularization with 
penalty term being of the type of an $l^1$-norm squared. It is well-known that such a penalty 
enforces sparsity, i.e., solutions to \eqref{eq:global_problem3} will lead to 
a $\bPhi$ that vanishes on many layers (if the real $\bPhi$ is sparse and  $\alpha$ is sufficiently large). 
This is the reason why \eqref{eq:global_problem} ($L = 1$) is not appropriate,  because the corresponding estimator 
usually has only values on  very few layers, which is not how a real atmosphere behaves. 
The introduction of the clustering structure of the layers and the weights $\bd$ allows for 
more flexibility and ``sparsifies'' less, depending on the weights.

%=======================================================
\subsection{Alternating Minimization Algorithm}\label{subsec:AMA}
%=======================================================

In Proposition~\ref{propmin} we  calculated the minimizers with respect to 
$\vcn,$ when $\bPhi$ is given. 
This suggests to apply an alternating minimization procedure for solving 
\eqref{eq:global_problem3}, where the approximate solutions 
$\bPhi_k,\vcn_k,$ $k = 1,\ldots $ are defined iteratively as 
\begin{align}\label{exactAMA}
\bPhi_{k+1}:= \argmin_{\substack{\bPhi  \in Y_n, \\ \bW{\vcn_k}^\dagger \bW{\vcn_k} \bPhi = \bPhi} } \Ffun(\bPhi,\vcn_k) \qquad  %\bPhi \in Y_n \\ \vcn\in \Gamma_N(\bd)\\ \bW{\vcn}^\dagger 
\vcn_{k+1}:= \argmin_{\vcn \in  \Gamma_N(\bd)} \Ffun(\bPhi_{k+1},\vcn_k). 
\end{align}
Note that $\vcn_{k+1}$ can be computed explicitly by 
\eqref{eq:vcnsol} (taking special care in case of $\bPhi$ being zero at a cluster). 
For numerical calculations, the evaluation of $\bW{\vcn}^\dagger$ is problematic as it can lead to numerical overflow
if $\vcn$ has very small entries. A not  uncommon  remedy is to regularize $\bWem$ to prevent it becoming near singular. 
Thus, we approximate problem \eqref{eq:global_problem3} by adding a small value $\epsilon$ to $\rho_l$ in $\bWem,$ 
i.e., $\bWeps{\vcn}:= \bW{\vcn+\epsilon}$ is used in place of $\bW{\vcn}$.  This has the additional effect that 
the constraint  $\bW{\vcn}^\dagger \bW{\vcn} \bPhi = \bPhi$ is trivially fulfilled and can be dropped.  

The computation of $\bPhi_{k+1}$ involves then a  quadratic 
optimization problem, hence, $\bPhi_{k+1}$ is given as  solution to the optimality conditions 
\begin{equation}
\label{eq:opt_condition}
\left(\bA^T  \cetap^{-1}  \bA + \alpha\cdot \bP(\vcn_{k})^{-1} \bCovl^{-1}\right)\bPhi_{k+1} =\bA^T \cetap^{-1}  \bphi,.
\end{equation}
with 
\begin{equation}\label{eq:cnmatrix} \bP(\vcn)^{-1}:= {\bWeps{\vcn}^*}^\dagger \bWeps{\vcn}^\dagger = 
\diag(\tfrac{1}{\rho_l +\epsilon})_{l=0}^L.  \end{equation}
However, 
as this system is usually of large scale, we modify this step and compute 
$\bPhi_{k+1}$ in \eqref{eq:opt_condition} by an approximate minimization using a finite number of steps of an 
iterative procedure like a gradient method. 
The resulting algorithm reads as follows:

 \begin{algorithm}[H]
 \caption{Alternating Minimization Algorithm in Blocks of Layers}\label{alg_altmin1}
 \begin{algorithmic}
 \State Parameter: $\alpha, \bd>0$, INNER, the number of inner gradient iterations, 
 $\bCovl^{-1},$ $\cetap^{-1},$ and a sequence of stepsizes $(\tau_{i})_k$ and initial 
 guesses $\vcn_{\rm init}, \bPhi_{\rm init}.$
 
 \noindent
 Set $\vcn_{1} = \vcn_{\rm init}$, $\bP(\vcn_{1})$ as in \eqref{eq:cnmatrix}, 
 $\bPhi_{1,0} = \bPhi_{\rm init}.$
% \State Calculate $\cetap^{-1}$ as in \eqref{eq:Cetap} \Comment{only if LGS are used} 
\For{$i=1, \dots$ \text{ until convergence} }
 %\State$\bPhi_0 := \bPhi_{i-1}$ 
\For{$k=1,\dots \text{INNER}$} \Comment{perform one or several steps of a gradient method}
\State $\bPhi_{i,k} = \bPhi_{i,k-1} + \tau_{i,k-1} \left( \bA^\ast \cetap^{-1}(\bphi - \bA \bPhi_{i,k-1}) - 
\alpha \cdot \bP(\vcn_{i})^{-1} \bCovl^{-1}\bPhi_{i,k-1}\right)$ %\Comment{without LGS}
 \EndFor
% \State$\bPhi_i := \bPhi_k$ 
\State$\bPhi_{i+1,0} := \bPhi_{i,k}$ 
% \Comment{Optional perform restart}
 \For {$n = 1, \dots , N$} 
 \For {$l  \in {\mathcal B}_n}$ 
 \State $\rho_l = d_n\cdot \norm{C_{\Phi}^{-1/2}\Phi_{i,k}^{(l)}}\Big/\dis\sum_{l\in {\mathcal B}_n}
 \norm{C_{\Phi}^{-1/2}\Phi_{i,k}^{(l)}}$  \Comment{update turbulence weights}
 \EndFor
 \EndFor
  \State$\vcn_{i+1} := (\rho_1,\ldots, \rho_N).$ 
\EndFor
\end{algorithmic}
\label{alg_altmin_block}
\end{algorithm} 
In Algorithm~\ref{alg_altmin1}, several implementation details are left out for simplicity. 
Of course, a standard convergence criterion for the outer loop is required. The parameter INNER defines 
the number of inner iterations. If formally $\text{INNER}=\infty$ (and assuming the gradient method converges),
then we obtain the  exact alternating
minimization in \eqref{exactAMA}. We preset the  parameter INNER to a fixed value as described below. 
The sequence of stepsizes $\tau_k$ can be chosen constant with the standard convergence constraint
$\tau_k <  \lambda_{\max}(\bA^\ast \cetap^{-1} \bA).$  However, we allow for more advanced non-constant stepsizes 
as well. 
%Special care has to be taken for the case that the weights $\rho_l$ approach $0$. 
%Note that then 
%$P(\vcn_{k-1})$ can become singular and the inner iteration steps are not well-defined any more. We 
%simply add in the definition of $P$ a small positive value $\epsilon$ to  $\rho_l$ to prevent this. 

Despite its deficiencies, the alternating minimization algorithm has one  big advantage over other methods 
like those that minimize problem \eqref{eq:global_problem2}: it is rather simple, and most important, 
it can be included easily into existing software. In fact, the inner iteration with the iteration 
variable $k$ for $\bPhi_k$ is exactly a standard gradient method for the atmospheric tomography problem when the 
turbulence strength is known as it is commonly used in AO. 
From the aspect of implementation, if a standard atmospheric tomography solver 
is already available, then the only additional procedure, namely calculating $\vcn$, is straightforward, 
and thus the alternating minimization method 
is a convenient choice. In particular, this applies to the dynamic (time-dependent) reconstruction problem, 
where the method can be easily adapted to.

Let us remark on the convergence of Algorithm~\ref{alg_altmin1}. Such alternating minimization 
procedures are well-known and can be viewed as  nonlinear variants of Gauss-Seidel iterations. 
A convergence proof (limit points of the sequence are stationary points) 
in the discretized (i.e., finite-dimensional case) can, e.g., be found in \cite[Proposition~2.7.1]{Bert} 
and the references therein.  Note, however, that for general functionals the alternating procedure does not have to 
converge, as it was shown by a counterexample of Powell \cite{pow73}.
The cited convergence result in \cite{Bert} applies to our case if we
restrict the problem to having all positive 
strength $\vcn \geq\epsilon >0$ (with a small parameter $\epsilon$, e.g., of the order of machine epsilon).  
Using these results and  the positivity restriction (which is just another way of introducing a regularization
of  $\bP(\vcn)^{-1}$ similar to \eqref{eq:cnmatrix})
it can be concluded that then  the exact iteration in \eqref{exactAMA} converges  to 
the unique minimum. 

We also mention that 
the transformation of problem  \eqref{eq:global_problem2} to problem \eqref{eq:global_problem3} and  the associated 
algorithms have also been applied in other contexts like in  variational image processing for nondifferentiable penalty functionals, 
e.g., for  the ROF-functional. There the corresponding alternating minimization method is known under 
the name of half-quadratic minimization \cite[Section 3.2.4]{Korn}, \cite{ChaLi} and also by the 
name of iteratively reweighted minimization \cite{RoWo}.

According to  Proposition~\ref{propmin3}, the two optimization problems \eqref{eq:global_problem} and
\eqref{eq:global_problem2} are equivalent and a similar alternating minimization algorithm for 
\eqref{eq:global_problem} can be designed. Since \eqref{eq:global_problem2} arises from a change of 
variables, it is not surprising that the corresponding iteration for $\bPhi_{i,k}$ is similar to that 
in Algorithm~\ref{alg_altmin1} but where the gradient direction is preconditioned by $\bP(\vcn_i).$ 
The reason why we prefer \eqref{eq:global_problem2} over  \eqref{eq:global_problem} is again the 
compatibility with existing code for the reconstruction of atmospheric turbulence, where usually $\bPhi$ 
is used as the unknown and not $\bPsi$.

In the next section we discuss an algorithm for solving  \eqref{eq:global_problem3}. As we will see, it
does not require an $\epsilon$-regularization of $\bWem$ and is hence more robust for 
the case of layers with vanishing  strength $\rho_i=0$.  It looses, however, the implementational benefit of 
Algorithm~\ref{alg_altmin1} of the compatibility with existing code.

%=======================================================
\subsection{Iterative shrinkage-thresholding algorithm} 
%=======================================================

The second algorithm for solving \eqref{eq:global_problem3} is based on $\Gfun$ defined in  \eqref{eq:reduced}, 
which includes a  nondifferentiable term (the $l^1$-norm squared). 
In particular, classical derivative-based method cannot be applied, but tools from convex analysis have to be 
employed.
%It is rather simple to implement but has no difficulties when the strength vector becomes $0$. 
%
The functional $\Gfun(\bPhi)$ can be split into two parts, a quadratic one $f$  and a convex nondifferentiable one $g$:
\begin{equation}\label{eq:split}  \Gfun(\bPhi) = f(\bPhi) + g(\bPhi) \, .
\end{equation}
%where we note that $g(\bPhi)$ is a convex nondifferentiable functional and $f$ is quadratic. 
There is an arsenal of methods to tackle the corresponding optimization problems, 
for instance, primal-dual algorithms \cite{ZC}, augmented Lagrangian methods \cite{He,Po}, Alternating Direction Method of 
Multipliers \cite{Boyd}, Bregman \cite{OBGXY} and split Bregman \cite{GoOs} methods,
just to mention a few. 

Since we want a simple yet  efficient one, we propose to use the  iterative shrinkage-thresholding algorithm (ISTA)
of Beck and Teboulle \cite{BeTe09} together with its  accelerated version FISTA. 
 ISTA is defined  as follows:
a parameter $\lip$  is fixed, which is set to be larger than the Lipschitz constant of $f$. 
In ISTA, a sequence of iterations $\bPhi_k$ for minimizing \eqref{eq:split}
is defined as
\begin{align*}
\bPhi_{k+1} :&= \text{prox}_{\lip^{-1}}(\bPhi_{k} - \tfrac{1}{\lip} \nabla f(\bPhi_k)), \qquad \text{where} \\
\text{prox}_{\lip^{-1}}({\bf u})&:= \argmin_{{\bf v}} g({\bf v}) + \tfrac{\lip}{2}\|{\bf u} -{\bf v}\|^2.  
\end{align*}
The function $\text{prox}_{\lip^{-1}}$ is called the proximal mapping. The efficiency of the method is based 
on the fact that the proximal mapping can often be calculated analytically. In case of $g$ being the $l^1$-norm, 
it amounts to applying a simple soft-shrinkage operator. Thus, to employ ISTA, we have to calculate the proximal mapping 
associated to the last term in \eqref{eq:reduced}. In order to do so, we first rewrite the optimization 
problem by a change of variables 
\begin{equation}\label{eq:relation} \myPsi^{(l)} := 
\Covl^{-1/2}\Phi^{(l)}  \quad \Longleftrightarrow \quad 
\Phi^{(l)} =  \Covl^{1/2} \myPsi^{(l)}, \qquad  \quad  l = 1,\ldots, L.
\end{equation}
In the new variables, problem \eqref{eq:reduced} reads as 
\begin{equation}
\label{eq:gdef}
\bmyPsi = \argmin_{\bmyPsi} \tilde{\Gfun}(\bmyPsi), \quad 
\tilde{\Gfun}(\bmyPsi) := \Gfun\left( \bCovl^{1/2} \bmyPsi \right), 
\quad 
\tilde{\Gfun}(\bmyPsi) = \tilde{f}(\bmyPsi) + \tilde{g}(\bmyPsi)
\end{equation}
where
\begin{equation*}
	\tilde{f}(\bmyPsi):= \norm{\cetap^{-1}\left({\bf A} \bCovl^{1/2} \bmyPsi  - \bphi\right)}^2_2 \quad {\rm and} \quad
\tilde{g}(\bmyPsi):= \alpha  \sum_{i=1}^L \frac{1}{d_i} \left(\sum_{l\in {\mathcal B}_i} \norm{\myPsi^{(l)}}\right)^2. 
\end{equation*}
Clearly the minimizers of $\Gfun$ and $\tilde{\Gfun}$ are related by \eqref{eq:relation}. 
The proximal mapping for $\tilde{g}$ is now separable over the clusters,
\[ \text{prox}_{\lip^{-1}}({\bf u}) = \argmin_{{\bf v} =(v^{(1)},\ldots, v^{(L)})} 
\sum_{i=1}^L \left[ 
\tfrac{\alpha}{d_i} 
\left( \sum_{l \in {\mathcal B}_i} \|{v}^{(l)}\| \right)^2 + 
 \tfrac{\lip}{2} \sum_{l \in {\mathcal B}_i} \|{u}^{(l)} -{v}^{(l)} \|^2 \right], \]
such that ${ v}^{(l)}$ can be computed clusterwise independently. The minimizer of 
the functional inside the brackets can  be found explicitly; see, e.g., \cite{SaLeAsLoNa13}.
%For such a weighted 
%sum of squares of the $l^1$-norm, this is well-known.
It involves a soft-thresholding within each cluster  
with threshold parameters $\gamma_\ind$, $\ind = 1,\ldots N$, 
\[ \text{prox}_{\lip^{-1}}(\bPsi)  = \max\left\{\Psi^{(l)} - \gamma_\ind ,0\right\} \cdot 
\frac{\Psi^{(l)}}{\|\Psi^{(l)}\|},  \quad \forall l \in {\mathcal B}_\ind,  \ind = 1,\ldots, N, \]
and $\gamma_\ind$ is calculated by the following algorithm, which we state 
in the variables $\Phi^{(l)} =  \Covl^{1/2} \myPsi^{(l)}$ for later reference.

\begin{algorithm}[H]
 \caption{Proximal mapping threshold parameter}
 \begin{algorithmic}
\State {\bf Input:} $\bPhi_k$, $\Covl$, $\alpha,$ $\lip$, $d_\ind,$ $\ind \in \{1,\ldots N\}$
 \For {$\ind = 1, \dots , N$} 
 \For {$l \in {\mathcal B}_m}$ 
\State Compute norms ${\bf n} \in \R^{|{\mathcal B}_\ind|}$, with ${\bf n} = \left(\norm{\Covl^{-\oh} \bPhi^{(l)}_k}\right)_{l \in {\mathcal B}_m}$
\EndFor
\State Sort norms ${\bf \tilde{n}} = \mbox{sort}({\bf n})$ with $\tilde{n}_1 > \dots > \tilde{n}_{|{\mathcal B}{\ind}|}$
\State Compute $l^\ast := \max \left\{l\in \{1,\dots ,|{\mathcal B}_{\ind}|\}: 
\tilde{n}_l - \frac{2\frac{\alpha}{\lip d_\ind}}{1+2 l \frac{\alpha}{\lip d_\ind}} \sum_{j=1}^{l} \tilde{n}_j >0\right\}$
\State $\gamma_\ind =  \frac{2\frac{\alpha}{\lip d_\ind}}{1+2l^\ast \frac{\alpha}{\lip d_\ind}} \sum_{j=1}^{l^\ast} \tilde{n}_j$
\EndFor

\noindent
{\bf Output:} $(\gamma_1,\ldots, \gamma_N)$
\end{algorithmic}
\label{alg_proxmapping}
\end{algorithm}

We  apply the ISTA iteration to $\tilde{\Gfun},$ yielding a sequence $\bPsi_{k}$. 
%Going back to the original variables, we find a sequence of approximate solution to 
%\eqref{eq:reduced} by $\bPhi_l^{k} =  C_{\bPhi_l}^{1/2} \bPsi_l^{k}.$ 
%Putting the pieces together and 
Upon a substitution to the original variables  
$\Phi_{k}^{(l)} =  \Covl^{1/2} \Psi_{k}^{(l)},$
we end up with the following Algorithm~\ref{alg:ista}.
Here we also include the possibility to have several inner gradient iterations with respect to 
$f,$ and  the stepsize~$\frac{1}{\lip}$ is fixed such that 
%such that the gradient of $f$ is Lipschitz continuous with  constant $\lip >0$, i.e.,
\begin{equation}
\label{eq:lip_const}
%\lip = \left\{\barr{ll}\lip_{\max}(\bA^\ast \bA)& %\mbox{without LGS}\\ 
\lip > \lambda_{\max}(\bA^\ast  \cetap^{-1}  \bA). %$\mbox{with LGS}\,.\earr\right.
\end{equation}

\begin{algorithm}[H]
 \caption{ISTA}\label{alg:ista}
 \begin{algorithmic}
 \State {\bf Input:} $\lip >0$, $\alpha, \bd>0$,  $\bCovl^{-1},$ $\cetap^{-1}$,
 initial guess $\bPhi_{\rm init}$
  \State Set $\vcn = \vcn_{\rm init}$, $\bPhi_{1,0} = \bPhi_{\rm init}$ 
% \State Calculate $\cetap^{-1}$ as in \eqref{eq:Cetap} \Comment{only if LGS are used}
\For{$i=1, \dots, $ until convergence }
%\IF{\#useGradientMethod}
\For{$k=1,\dots, $ INNER} \Comment{perform one or several steps of a gradient method}
\State $\bPhi_{i,k} = \bPhi_{i,k-1} + \frac{1}{\lip}\cdot \bCovl \left( \bA^\ast \cetap^{-1}(\bphi - \bA \bPhi_{i,k-1})\right)$ %\Comment{without LGS}
% \State $\bPhi_{i,k} = \bPhi_{i,k-1} + \frac{1}{\lip} \cdot  \bCovl  \left( \bA^\ast \bpi \cetap^{-1} \bpi(\bphi - \bA \bPhi_{i,k-1})\right)$ \Comment{with LGS}
 \EndFor

 \State Calculate threshold $\gamma_\ind$ for $\ind =1,\ldots, N$
 via Algorithm~\ref{alg_proxmapping} with $\bPhi_{i,k}$
 %\For {$l = 1, \dots , L$} 
 \State  $\Phi_{i,k}^{(l)} = \max\left(\norm{\Covl^{-1/2}\Phi_{i,k}^{(l)}} - 
 \gamma_\ind, 0\right) \cdot\frac{\Phi_{i,k}^{(l)}}{\norm{\Covl^{-1/2}\Phi_{i,k}^{(l)}}}, \quad \forall l \in {\mathcal B}_\ind, 
 \ind =1,\ldots, N$
 \Comment{shrinkage}
% \EndFor
 \State Set $\bPhi_{i+1,0} = \bPhi_{i,k}$
% \For {$l = 1, \dots , L$} 
% \State $\rho_l = d_n\cdot \dis\norm{C_{\Phi^{(l)}}^{-1/2}\Phi_{i,k}^{(l)}}\Big/\dis\sum_{j = 1}^L \norm{C_{\Phi^{(j)}}^{-1/2}\Phi_{i,k}^{(j)}}$  \Comment{update turbulence weights}
% \EndFor
 
\EndFor
\State {\bf Output:} $\bPhi_{i,k}$
\end{algorithmic}
\label{alg_ista1}
\end{algorithm}

The original version of ISTA has $\text{INNER}=1.$ Note that this algorithm has no 
problem if $\norm{C_{\bPhi^{(l)}}^{-1/2}\bPhi_{i,k}^{(l)}} = 0$ as the threshold is then 
simply set to $0$ in the shrinkage step. As before, the turbulence strength weights
$\vcn$ can be found in any outer iteration step by formula \eqref{eq:vcnsol}.

%---------------------------------
%---------------------------------
In order to accelerate convergence of ISTA, a
\textit{fast iterative shrinkage-thresholding algorithm} (FISTA) has been 
introduced in \cite{BeTe09}. The method is very similar to ISTA except that
it involves a two-step procedure (or alternatively two iteration variables). 
Only the iteration for $\bPhi$ is altered while the shrinkage step stays the same. 
The algorithm exhibits a similar computational cost per iteration than ISTA, 
but is faster in general. 
%implies a better convergence rate of $O(\frac{1}{i^2})$, instead of the convergence rate $O(\frac{1}{i})$ of ISTA.
There is no problem in applying the FISTA acceleration concept to Algorithm~\ref{alg:ista};
as this is quite standard, we omit its description.

%============================================
\section{Numerical results}
\label{sec:numerics}

%============================================
%--------------------------
\subsection{Simulation setup}
%--------------------------

Our simulation setting consists of a 10m telescope with a circular aperture. We assume a 
guide star asterism with 6 natural guide stars arranged in a circle with a radius of 1.5 arcmin; see Figure \ref{fig:tel_geo}. 
%The wide field of view of 3 arcmin is chosen in order to emphasise the need of a good tomographic reconstruction for a decent overall correction.
In our simulations, we assume 6 corresponding Shack-Hartmann wavefront sensors. % with a 20$\times$20 discretization.
As discussed above, we consider only the tomography step, i.e., take the incoming wavefronts (calculated by an arbitrary wavefront reconstructor) as given.  
The correction is made by means of one deformable mirror. %with a 21$\times$21 discretization (which corresponds to a 0.5m actuator spacing)

\begin{figure}[H]
\centering
\includegraphics[scale = 0.75]{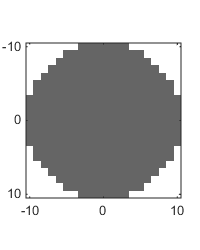} \hspace*{1cm}
\includegraphics[scale = 0.75]{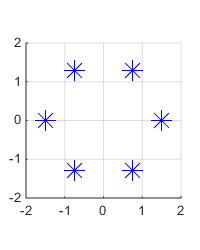}
\caption{Wavefront data discretization ($21\times 21$) on telescope aperture in grey (left). Guide star asterism with 6 natural guide stars in a circle of 1.5 arcmin radius (right).}
\label{fig:tel_geo}
\end{figure}

This mimics a Laser Tomography AO or a Multi-Object AO system and is well suited for evaluating the purely tomographic performance of the algorithms. The projection of the reconstructed layers onto the telescope aperture gives the optimal shape of the deformable mirror, thus, no sophisticated fitting step as, e.g., in Multi-Conjugate AO \cite{Fusco,RaRo12,HeYu13,RaRaYu15}, is needed between tomographic reconstruction and quality evaluation. 

We consider two prominent atmospheric models: the 9-layer profile and the 40-layer profile, provided by the ESO. For comparison of the two profiles, the relative (logarithmic) densities in each Voronoi interval \cite{Auzi14}, i.e. the absolute weights $\rho_l$ divided by the interval length, are depicted in Figure \ref{fig:atms}. Additionally, the corresponding layer heights $h_l$ are marked, $l=1,\dots ,L$ with $L=9$ and $L=40$ respectively.

\begin{figure}[h!]
\centering
\includegraphics[scale = 0.75]{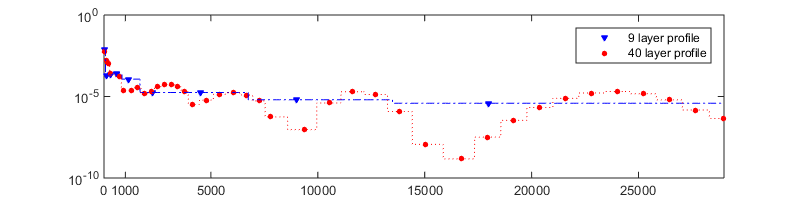}

\caption{Height (in meter) vs. relative (log) density of atmospheric profiles: 9-layer (blue, triangle), 40-layer (red, circle).}
\label{fig:atms}
\end{figure}

We simulate a layered atmosphere following the von Karman model \eqref{eq:cphi}. 
By projection onto the telescope aperture in the center direction the simulated screen is obtained. The incoming wavefronts 
in guide star directions serve as input data for our reconstruction algorithms. 
After the reconstruction, the atmospheric layers are projected onto the telescope aperture in 
the center direction in order to get the shape of the deformable mirror. 

 Image quality is evaluated over the whole 3 arcmin field of view as well as in the center direction 
 by means of the short-exposure Strehl ratio $S_\varphi$ \cite{LeVeKo05}, a value between 0 (worst) and 1 (best), 
 which we approximate via the Marechal criterion (e.g. \cite{Ro99})
 \begin{equation}
 \label{eq:marechal}
 S_\varphi \approx \exp\left(-\frac{1}{|\Om_D|}\norm{\overline{\varphi} - \varphi}^2_{L^2(\Om_D)}\right)\,,
 \end{equation}
 with $\dis\overline{\varphi} =\frac{1}{|\Om_D|}\int_{\Om_D} \varphi(r)\, dr\, $ the average phase over the aperture $\Om_D$. Furthermore, we calculate the relative error in directions $\ag$
 \begin{equation}
 \label{eq:rel_error}
 \eps_{\ag} = \frac{\norm{\Ao{g} \bPhi^{\rm orig}}_2 - \norm{\Ao{g} \bPhi^{\rm rec}}_2}{\norm{\Ao{g} \bPhi^{\rm orig}}_2}\,.
 \end{equation} 
 
 In the following, we only consider one time step, which spares us any concerns about
 temporal control and gain tuning. We refer to future work for a full simulation over more time steps.% in order to evaluate the performance and stability in closed loop.
 
%--------------------------
\subsection{Alternating minimization results}
%-------------------------- 

In Figure \ref{fig:altmin_vs_grad} we demonstrate the qualitative performance of the alternating minimization algorithm \ref{alg_altmin1}. 
We sample 512 realizations of the 9-layer atmospheric profile and compute the average radial Strehl ratio of our method and the standard gradient method \cite{SaRa15,RaSaYu14}. The alternating minimization algorithm outperforms the gradient method even if the latter is configured with exact layer weights.

\begin{figure}[h!]
\centering
\includegraphics[scale = 0.75]{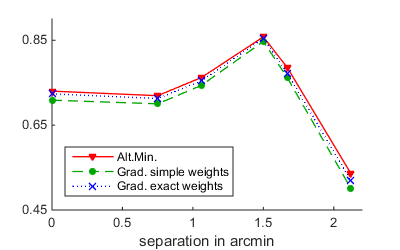}%sim_rec_diff}
\caption{Alternating Minimization Algorithm \ref{alg_altmin1} (red, triangle) vs. gradient method on fixed layer altitudes and exact weights (blue, cross) or simple weights $\vcn = (0.5, \frac{0.5}{8}, \dots , \frac{0.5}{8})$ (green, circle): Radially averaged strehl ratio over 512 realizations of the 9-layer profile vs. separation from center direction (in arcmin).}% 100 (outer) iterations, 
%INNER = 10, $\alpha = 0.1$, all 9 layers in one cluster (i.e. $N=1$), ground-layer enforcing initial guess 
%$\vcn_{\rm init} = (0.5, \frac{0.5}{8}, \dots , \frac{0.5}{8})$. % stepsize normal
%}
\label{fig:altmin_vs_grad}
\end{figure}

The Strehl values are evaluated in 25 directions, i.e. in a $5\times 5$ square over the field of view. In Figure \ref{fig:altmin_vs_grad}, the radial average is displayed over the separation in arc minutes from the center direction. At the guide star radius, Strehl ratios are largest, whereas due to the big field of view, the quality in the center is significantly lower. This is a typical behaviour for wide field of view AO: comparable results but for bigger telescope apertures were obtained, e.g. in \cite{RoRa13}.
 
Tests in Figure \ref{fig:altmin_vs_grad} have been performed with 100 (outer) iterations and 10 INNER iterations, a regularization parameter $\alpha = 0.1$ and a ground-layer enforcing initial guess  $\vcn_{\rm init} = (0.5, \frac{0.5}{8}, \dots , \frac{0.5}{8})$. In the following, we stick to the choice of iteration number and starting value for the turbulence weights. 
A higher number of iterations improves quality by less than 1\% of Strehl and relative error and does not influence 
the reconstructed turbulence weights significantly. Similarly, the choice of initial weights only slightly influences 
the performance of the algorithm. With  uniform starting values, i.e. $\vcn_{\rm init} = (\frac{1}{9},\dots ,
\frac{1}{9})$, similar convergence behaviour but slightly lower quality was obtained.
The large weight on the ground layer is physically motivated as most of the turbulence appears there.

In Figure \ref{fig:altminGrad} we depict results for Algorithm \ref{alg_altmin_block} with three clusters, i.e. $N=3$ (bottom), and without clustering,
i.e. $N = 1$ (top). The left column shows the reconstructed turbulence weights on the nine layers, along with the weights used for simulation. In the right column, the Strehl ratio (in center direction and averaged over the field of view) and the relative error for varying regularization parameters $\alpha$ are depicted. Please note, that the prescribed interval strength $d_m$ per cluster is displayed as line at value $d_m$ over all layers in cluster ${\mathcal B}_m$ in all following figures. 
The number and sizes of the clusters ${\mathcal B}_m$ and the corresponding weights $d_m$ can vary. Here, three clusters are modelled
based on the idea of reducing the reconstruction effort to three instead of nine layers. Different choices of clusters 
and corresponding weights were tested yielding similar results.

\begin{figure}[h!]
\centering
\includegraphics[scale = 0.75]{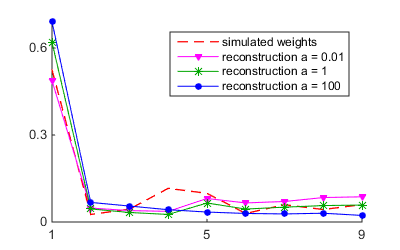}
\includegraphics[scale = 0.75]{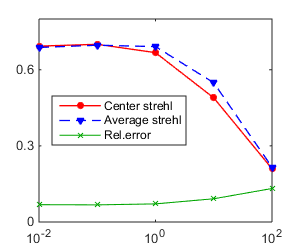}\\
\includegraphics[scale = 0.75]{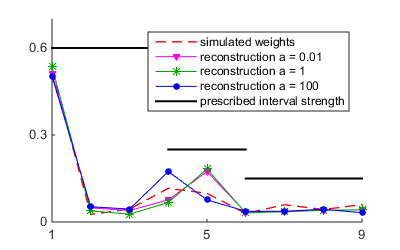}
\includegraphics[scale = 0.75]{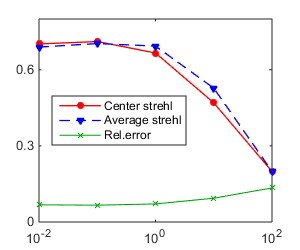}
\caption{Alternating Minimization Algorithm \ref{alg_altmin1}: Turbulence weights vs. layers (left), Strehl and error vs. regularization parameter (right). 1 cluster (top),  3 clusters (bottom). 9-layer atmosphere, 100 (outer) iterations, INNER = 10.}%, varying $\alpha = 0.01, 0.1, 1, 10, 100$.} % stepsize normal

\label{fig:altminGrad}
\end{figure}

We can observe that even for large regularization parameters $\alpha$ no sparse profile is obtained. This stems 
from the slow convergence  in case of  turbulence weights $\rho_l$ being close to zero.  The direct solution of the optimality condition for $\bPhi$ in equation \eqref{eq:opt_condition} would formally be identical to choosing INNER = $\infty$.
The result for solving the optimality condition 
by direct matrix inversion with ${\bP}^{-1}(\vcn) = \diag \left(\frac{1}{\rho_l + \epsilon}\right)$ and $\epsilon \sim 10^{-10}$ for stabilization, 
is depicted in Figure \ref{fig:altmin_direct}. In this case, we can observe that the sparsity of the solution is more pronounced. A similar result as in Figure~\ref{fig:altmin_direct} can be obtained by considering the alternating minimization procedure for \eqref{eq:global_problem} by multiplying $\bP(\vcn_i)$ to the update in the gradient step in 
Algorithm~\ref{alg_altmin1}, as discussed at the end of Section~\ref{subsec:AMA}. %This yields similar results as using a direct solver as seen in Figure~\ref{fig:altmin_direct}.

\begin{figure}[h!]
\centering
\includegraphics[scale = 0.75]{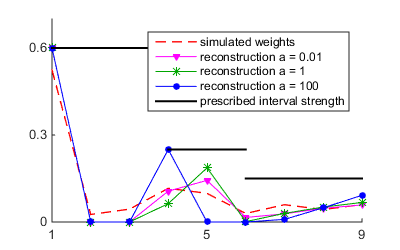}%\\
\includegraphics[scale = 0.75]{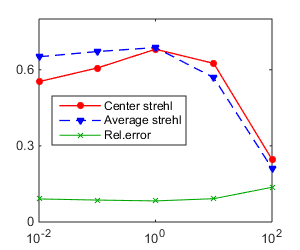}
\caption{Alternating Minimization Algorithm \ref{alg_altmin1} with direct solver for the optimality condition \eqref{eq:opt_condition}: Turbulence weights vs. layers (left), Strehl and error vs. regularization parameter (right). 9-layer atmosphere, 100 iterations, 
3 clusters.}%, ground-layer enforcing initial guess $\vcn_{\rm init} = (0.5, \frac{0.5}{8}, \dots , \frac{0.5}{8})$, varying $\alpha = 0.01, 0.1, 1, 10, 100$.} % stepsize normal
\label{fig:altmin_direct}
\end{figure}

%As discussed at the end of Section~\ref{subsec:AMA}, we can consider the alternating minimization procedure also 
%for \eqref{eq:global_problem} by multiplying $\bP(\vcn_i)$ to the update in the gradient step in 
%Algorithm~\ref{alg_altmin1} such that no numerical instabilities occur.
%This yields similar results as using a direct solver as seen in Figure~\ref{fig:altmin_direct}.
%However, more iterations are needed for convergence.

Therefore, we can conclude that the alternating minimization with the clustering approach approximates the true turbulence profile and with increasing regularization parameter also enforces sparsity. % when taking care of numerical instabilities.
 Moreover, the algorithm is easily integrated into existing reconstruction methods.

%--------------------------
\subsection{ISTA results}
%--------------------------
Compared to the alternative minimization approach, we can obtain a sparse turbulence profile with fewer iterations by utilizing the  Iterative Shrinkage-Thresholding Algorithm (ISTA). Without the clustering, ISTA tends to recover all turbulence strength on the ground layer.
As noted above, this is physically realistic. However, for wide field tomography problems higher vertical resolution is required. 
This can be guaranteed by introducing clusters.

\begin{figure}[h!]
\centering
%\begin{minipage}{0.61\textwidth}
\includegraphics[scale = 0.75]{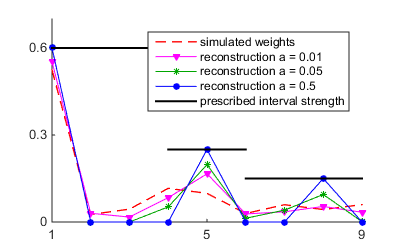}
%\end{minipage}
%\begin{minipage}{0.38\textwidth}
\includegraphics[scale = 0.7]{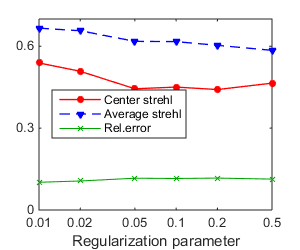}
%\\
%\includegraphics[scale = 0.7]{ista9_quali0}
%\end{minipage}
\caption{ISTA: Turbulence weights vs. layers (left), Strehl and error vs. regularization parameter (right). 9-layer atmosphere, 10 (outer) iterations, INNER = 10, 
3 clusters.}%,  $\lambda = 4$.}% ground-layer enforcing initial guess $\vcn_{\rm init} = (0.5, \frac{0.5}{8}, \dots , \frac{0.5}{8})$,, varying $\alpha = 0.01, 0.02, 0.05, 0.1, 0.2, 0.5$.}
\label{fig:ista_alpha}
\end{figure} 
 
In Figure \ref{fig:ista_alpha} we demonstrate the results for ISTA with varying parameter $\alpha$. Recall that $\alpha$ influences the threshold for the shrinkage in Algorithm \ref{alg_proxmapping}. 
The clustering approach approximates the true turbulence profile for small parameters $\alpha$ and with increasing $\alpha$ enforces sparsity, where only one layer per cluster has a nonzero turbulence weight. 

%We can observe that only a low number of outer iterations is needed for a good center as well as average reconstruction, i.e. 10 (outer) iterations along with 10 INNER iterations suffice. 

For ISTA, 10 (outer) iterations along with 10 INNER iterations are already sufficient to reach a sparse 
turbulence profile and good quality. Compared to the alternating minimization algorithm, the quality is 
lower, in particular in the center. However, a sparse profile can be reached and the number of iterations is considerably lower.

% Increasing the number of iterations does not improve Strehl ratio.

\begin{figure}[h!]
\centering
\includegraphics[scale = 0.7]{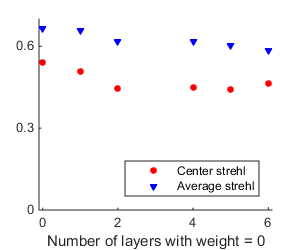}
\caption{ISTA as in Figure \ref{fig:ista_alpha}: Strehl value vs. number of layers with zero weights. }
\label{fig:ista_alpha_lay0}
\end{figure} 

In Figure \ref{fig:ista_alpha_lay0} we can observe that the best 
reconstruction quality is obtained when all weights are non-zero. Naturally, the quality 
decreases when some weights are set to zero. However, reducing the number of non-zero turbulence weights, i.e.,
the number of reconstruction layers, e.g., from seven to three, does not alter the quality significantly.  Thus, ISTA seems to be a promising approach to answering the question which layer altitudes are the most relevant for the reconstruction. 

For the sake of convergence speed, the covariance matrix $\bCovl$ of the turbulence statistics 
can be omitted in the gradient step. %For the numerical results we skip the application of $C_{\bPhi}$ on the update in the gradient loop of ISTA, Algorithm \ref{alg_ista1}, and FISTA. 
With a proper choice of $\lip$ and the number of ISTA iterations, both variants --- with and without $\bCovl$ --- 
yield very similar results and show comparable convergence behaviour.

The stepsize has to fulfill $\frac{1}{\lambda} \geq \frac{1}{\norm{\bA^\ast \bA}}
= \frac{1}{\lip_{max}(\bA^\ast \bA)} \approx 0.022 $. In the examples presented here, 
we chose a fix stepsize of 0.25, i.e. $\lambda = 4$, which leads to  fast convergence. 
 A more sophisticated stepsize choice, such as the classical steepest descent stepsize % with an acceleration, as in \cite{SaRa15},
  can improve quality, in particular, the center Strehl ratio by several percent and the average Strehl ratio by a few percent. %5\% and 1-2\%.

\begin{figure}[h!]
\centering
\includegraphics[width = 0.47\textwidth]{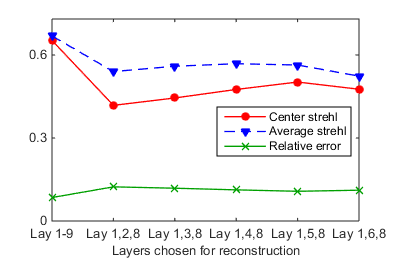}

\caption{ISTA chooses the best reconstruction layer altitudes: Strehl and error vs. reconstruction layers. Gradient method on fixed layer altitudes and heights, 100 iterations, reconstruction on all 9 layers (leftmost) and on 3 layers 
(cluster 1 and 3 fixed, i.e. layers 1 and 8, and varying layer in cluster 2).} 
\label{fig:bestof}
\end{figure}

In Figure \ref{fig:bestof} we demonstrate that ISTA recovers the qualitatively optimal reconstruction layer altitude per cluster. 
We run ISTA as in Figure \ref{fig:ista_alpha} with $\alpha = 0.5$, and determine the 
preselected candidates for non-zero reconstruction layers, i.e., layers 1, 5 and 8. 
Then we apply a standard gradient method \cite{SaRa15,RaSaYu14} with a fixed number of layers  and altitudes
and heights set to different combinations.  We fix layer 1 and 8 and vary the reconstruction layers 
in cluster 2 (layers 2 to 6). One can  observe, that the reconstruction layer 5 --- chosen as optimal both by ISTA and the 
alternating minimization for 
cluster 2 --- yields the best results compared to a reconstruction with  layers 2, 3, 4, and 6 instead.
This clearly demonstrates the potential of our algorithms to automatically select the ``best'' model (layers). 

Furthermore, Figure \ref{fig:ista_blockvar} shows that the choice of reconstruction layer 5 by ISTA and  the 
alternating minimization for 
cluster 2 is consistent with the variation of cluster sizes.

\begin{figure}[h!]
\centering
\includegraphics[scale = 0.75]{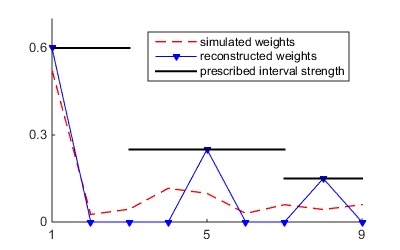}
\includegraphics[scale = 0.75]{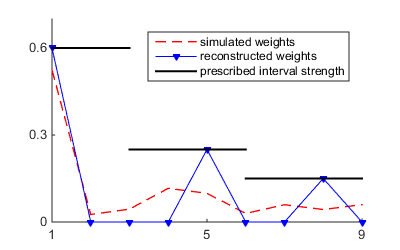}
\caption{Block variation for ISTA (settings as in Figure \ref{fig:ista_alpha}): Turbulence weights vs. layers.}
\label{fig:ista_blockvar}
\end{figure}

Similar results can be obtained for the 40-layer atmosphere. In Figure \ref{fig:ista_40lay}, we can observe, that 
already 50 (outer) iterations suffice to guarantee convergence to one layer 
per cluster. For $\alpha$ large enough, i.e.,  greater than 0.05, the convergence to the non-zero 
layers 1, 2, 12, 24 and 37, is robust with respect to varying 
cluster sizes and weights. As for the 9-layer atmosphere, the quality is very stable over the number of iterations.

\begin{figure}[h!]
\centering
\includegraphics[scale = 0.75]{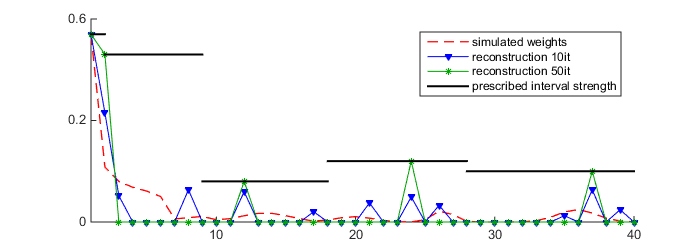}
\includegraphics[scale = 0.75]{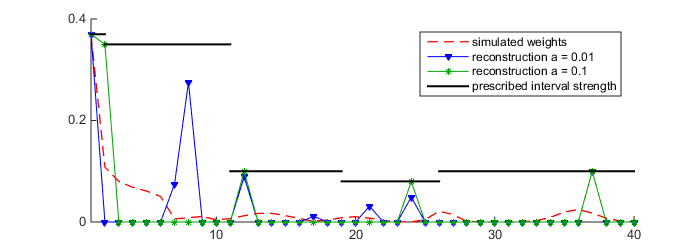}

\includegraphics[scale = 0.75]{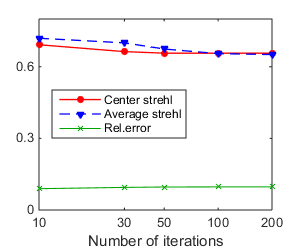}
\includegraphics[scale = 0.75]{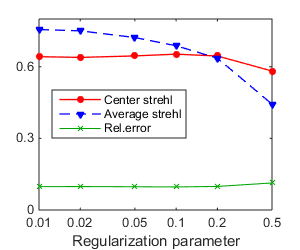}

\caption{ISTA on 40-layer atmosphere: Turbulence weights vs. layers (top), Strehl and error vs. number of (outer) iterations (bottom, left), vs. regularization parameter (bottom, right). INNER = 10, 5 clusters.}%,  ground-layer enforcing initial guess $\vcn_{\rm init} = (0.3, \frac{0.7}{39}, \dots , \frac{0.7}{39})$, $\lambda = 4$: varying $\alpha = 0.01, 0.02, 0.05, 0.1, 0.2, 0.5$ (middle and bottom right), varying (outer) iterations 10, 30, 50, 100, 200 (top and bottom left).}
\label{fig:ista_40lay}
\end{figure}

Using the FISTA-variant of Algorithm \ref{alg_ista1},
a speed up in convergence can be obtained. However, in these examples, 
the reduction in the computational effort is rather small, as only few iterations are needed anyway. 
However, this speed up could be a crucial point for systems with finer resolution such as the ELT.

%While the alternating minimization yields a quality enhancing reconstruction and approximates the true turbulence profile, ISTA permits a model reduction via a sparse solution.

%Alt.Min. approximates turbulence profile better, i.e., quality in particular in the center gets better than in ISTA. But it needs more iterations, and problems with numerical instabilities.

%--------------------------
%\subsection{Computational Complexity} 
%--------------------------
The computational cost of one gradient step is determined by 
the operators $\bA, \bA^\ast$ and the application of $\bCovl^{-1}$. The numerical effort of 
$\bA$ and $\bA^\ast$ mainly stems from bilinear interpolation. The application of the inverse covariance matrix 
$\bCovl^{-1}$ is the most costly step in both Algorithms~\ref{alg_altmin_block} and \ref{alg:ista}.  
Note, that in Algorithm~\ref{alg_altmin_block}, this has to be performed in each INNER step, i.e., 
$k\cdot i$ times, while in ISTA, it is only needed in the shrinkage step, i.e., $i$ times. 
By thresholding, the dense matrices $C_{\Phi}^{-1}$ can be sparsified without loss of qualitative performance. 
However, more sophisticated methods, as, e.g., in \cite{YuHeRa13,YuHeRa13b}, are needed to efficiently apply 
the method for larger systems with finer resolution. 
Both algorithms can be parallelized on the reconstruction layers, as well as in the  guide star directions.

\begin{table}[h!]
\centering
\begin{tabular}{|l||c|c||c|c|}
\cline{2-5}
% & \multicolumn{2}{c|}{FLOP} \\
\multicolumn{1}{c|}{}& $L$ & cost per it & $L$ & cost per it \\
% & \multicolumn{4}{|c|}{overall complexity &  per process\\
\hline
\hline
original atmosphere & 40 & $3908n$  &  9 & $839n$\\
\hline
downsampled atmosphere & 5 & $443n$ & 3 & $245n$ \\
\hline
speed up factor & \multicolumn{2}{|c||}{8.8} &  \multicolumn{2}{|c|}{3.4} \\
\hline
\end{tabular}
\caption{Upper bounds of the computational cost estimates per standard gradient iteration, one iteration costs $(16G+3)nL + (2-9G)n$, $G=6$ guide stars and $L$ the number of reconstruction layers.}
\label{tab_comp}
\end{table}

Table~\ref{tab_comp} shows the computational cost estimates and speed 
up factors for a standard gradient iteration  \cite{SaRa15} (with fixed turbulence heights and weights). 
With the resulting downsampled turbulence profiles from the alternating minimization algorithm or ISTA, 
we could reach speed up factors of 8.8 and 3.4 for the 40-layer and 9-layer atmosphere, respectively. 
Again, this should serve as a justification of the proposed methods as a tool for joint identification and 
model reduction.

%============================================
\section{Conclusion}
In this paper, we presented a new  reconstruction algorithm for  atmospheric tomography  
that includes  an automatic model optimization of the reconstruction profile. 
We derived the corresponding Bayesian model
and presented two different algorithms to solve the joint optimization with respect to the atmospheric 
layers and their turbulence weights. Our numerical results suggest that both algorithms choose optimal 
reconstruction heights which leads to a significant speed up while still results of good quality  are obtained.
Moreover, our algorithms can be easily adapted for a two-step method, 
such as \cite{HeYu13}, with wavefront sensor measurements as input data instead of incoming wavefronts.

We believe that our method is very promising especially for wide field of view AO systems 
as the approach unites the idea of conventional compression algorithms for a 
layered atmosphere and the tomographic reconstruction itself.  
The numerical experiments presented here indicate that the methods are not yet 
feasible for real-time usage in present-day AO systems with many degrees of freedom due to 
some rather time-consuming steps like the application of the turbulence statistics. 
However, our model reduction approach  can still be used offline for profile optimization, running in 
parallel to a standard  real-time reconstructor. The updated and optimized profile 
can then be included into the reconstructor  whenever available. 
%
%Then a standard reconstructor can use this profile
%or the next time steps until another iteration of our algorithm is performed and a new, optimal reconstruction profile obtained. 
This allows to adapt on the fly to changing atmospheric conditions.
%\bibliographystyle{abbrv} % apa abbrv
%\bibliography{../main,../eso_sorted,./additional}
%\bibliography{TPE}
\def\cprime{$'$}

\end{document}